\documentclass[a4paper]{article}
\usepackage[latin1]{inputenc} 
\usepackage[T1]{fontenc} 
\usepackage{RR,RRthemes}
\usepackage{hyperref}

\usepackage{graphicx}
\usepackage{amsmath}
\usepackage{amssymb}
\usepackage{graphics}
\usepackage{color}
\usepackage{epsfig}
\usepackage[latin1]{inputenc} 

\newtheorem{lemma}{Lemma}[section]

\newtheorem{example}{Example}[section]
\def\Box{\diamond}
\newenvironment{proof}{\begin{trivlist}
                       \item[]{\bf Proof.}
                       \hspace{0cm}}{\hfill $\Box$
                       \end{trivlist}}
                     {\end{enumerate}}

\newcounter{statement}
\newcounter{algorithm}
\newcounter{secalgorithm}[section]
\newcounter{subalgorithm}[subsection]
{\catcode`\;=\active%
\gdef\pointvirgule{\catcode`\;=\active%
\def;{\unskip\kern 2pt\string;~~~\ignorespaces}}}
\def\refstepalgorithms{\stepcounter{subalgorithm}\stepcounter{secalgorithm}
\refstepcounter{algorithm}}
{\end{tabbing}}
\newenvironment{algorithm*}[1]{\pointvirgule\setcounter{statement}{0}
\def\N>{\stepcounter{statement}\thestatement\>}\def\B>{\hskip0.1em\vrule\>}
\def\[##1]{{{##1}}}
\begin{tabbing}
++\=++\=++\=++\=++\=++\=++\=++\=++\=++\kill
{{\sc Algorithm}} : {#1}\\*[1mm]}%
{\end{tabbing}}

{\end{tabbing}}

\newenvironment{falgorithm*}[1]{\pointvirgule\setcounter{statement}{0}
\def\N>{\stepcounter{statement}\thestatement\>}\def\B>{\hskip0.1em\vrule\>}
\def\[##1]{{{##1}}}
\begin{tabbing}
++\=++\=++\=++\=++\=++\=++\=++\=++\=++\kill
}%
{\end{tabbing}}
\newenvironment{*algorithm*}{\pointvirgule\setcounter{statement}{0}
\def\N>{\stepcounter{statement}\thestatement\>}\def\B>{\hskip0.1em\vrule\>}
\def\[##1]{{\bf##1}}
\begin{tabbing}++\=++\=++\=++\=++\=++\=++\=++\=++\=++\kill}{\end{tabbing}}
{\end{tabbing}}
\newenvironment{ualgorithm*}[1]{\pointvirgule\def\B>{\hskip0.1em\vrule\>}
\def\[##1]{{\bf##1}}
\begin{tabbing}
++\=++\=++\=++\=++\=++\=++\=++\=++\=++\=\kill
\underline{{\sc Algorithm}}: {#1}\\*[1mm]}%
{\end{tabbing}}

\catcode`@=11
\def\@comment[#1]#2{$\{$\hbox to #1{#2 \dotfill}$\}$}
\def\comment{\@ifnextchar[{\@comment}{\@comment[10.5cm]}}
\catcode`@=12

\def\eps{\ifmmode\epsilon\else$\epsilon$\fi}
\def\veps{\ifmmode\varepsilon\else$\varepsilon$\fi}

\newsavebox\commentbox
{\egroup\endgroup}

\usepackage{amsfonts}

\def\IR{\mathbb{R}}
\def\IC{\mathbb{C}}

\def\og{\leavevmode\raise.3ex\hbox{$\scriptscriptstyle\langle\!\langle$}}
\def\fg{\leavevmode\raise.3ex\hbox{$\scriptscriptstyle\rangle\!\rangle$}}

\def\romain#1{\uppercase\expandafter{\romannumeral#1}}

\def\trait{\hskip 0.5mm \raise -6pt\vbox{\vrule height 18pt}\>}
\def\traitl{\hskip 0.5mm \raise -11pt\vbox{\vrule height 25pt}\>}
\def\hfl#1#2{\smash{\mathop{\hbox to 20mm {\rightarrowfill}}
\limits^{\scriptstyle #1}_{\scriptstyle #2}}}

\def\SS#1#2{{#1}\kern .15em{#2}}
\def\into#1{\S\kern .15em\ref{#1}}

\catcode`@=11
\def\mycaption#1#2#3{\par\@bsphack\vbox{\def\@captype{#1}\caption{#3}}
{\if@filesw {\xdef\@gtempa{\write\@auxout{\string\newlabel{#2}
{{\csname p@#1\endcsname\csname the#1\endcsname}{\thepage}}}}}}
\@gtempa \if@nobreak\ifvmode\nobreak\fi\fi\fi\@esphack}
\catcode`@=12

\catcode`@=11
\newcount\@hcaptype
\def\@@hcaption[#1]#2#3{\par\@bsphack\vbox{\def\@captype{#1}\caption{#3}}
{\if@filesw{\xdef\@gtempa{\write\@auxout{\string\newlabel{#2}
{{\csname p@#1\endcsname\csname the#1\endcsname}{\thepage}}}}}}
\@gtempa\if@nobreak\ifvmode\nobreak\fi\fi\fi\@esphack}
\def\@hcaption#1#2{\ifnum\@hcaptype=0\@@hcaption[figure]{#1}{#2}
\else\@@hcaption[table]{#1}{#2}\fi}
\def\hcaption{\@ifnextchar[{\@@hcaption}{\@hcaption}}
{\begin{trivlist}\@hcaptype=0\centering\item[]}%
{\end{trivlist}}
{\begin{trivlist}\@hcaptype=1\centering\item[]}%
{\end{trivlist}}
\catcode`@=12

\def\det{{\rm det}}
\def\Arg{{\rm Arg}}
\RRdate{October 2011}

\RRauthor{Emmanuel Kamgnia \thanks{University of Yaounde I, Cameroun.}%
  \thanks[sfn]{This work was pursued within the team MOMAPPLI of the LIRIMA (Laboratoire International de Recherche en Informatique et Mathématiques Appliquées).}%
  \and
Bernard Philippe \thanks{Centre INRIA de Rennes Bretagne Atlantique}%
\thanksref{sfn}}
\RRtitle{Dénombrement de valeurs propres \\ dans le plan complexe}
\RRetitle{Counting eigenvalues in domains \\ of the complex field}
\titlehead{Counting eigenvalues}
\authorhead{E. Kamgnia and B. Philippe}

\RRabstract{A procedure for counting the number of eigenvalues of a matrix in a region surrounded by a closed curve is presented.  It is based on the application of the residual theorem.  The quadrature is performed by evaluating the principal argument of the logarithm of a function.  A strategy is proposed for selecting a path length that insures that the same branch of the logarithm is followed during the integration.  Numerical tests are reported for matrices obtained from conventional matrix test sets.}
\RRresume{Nous proposons une approche pour compter le nombre de valeurs propres d'une matrice, situées dans un domaine borné du plan complexe. L'approche est fondée sur l'application du théorème des résidus. L'intégration nous ramène à l'évaluation de l'argument principal du logarithme d'une fonction.  Nous proposons une stratégie pour le choix du pas qui permette de rester sur la même branche du logarithme pendant l'intégration. Des résultats numériques sont obtenus à partir de tests conduits sur des matrices tirées d'ensembles classiques de matrices.}
\RRkeyword{Eigenvalues, resolvent, determinant, complex logarithm}
\RRmotcle{Valeurs propres, résolvante, déterminant, logarithme complexe}
\RRprojet{Sage}
\RRdomaine{5a} 
\RRthemeProj{sage} 
\RRdomaineProjBis{sage} 
\URRennes %

\begin{document}
\RRNo{7770}
\makeRR   

\newpage
\section{Introduction}
\indent The localization of eigenvalues of a given matrix $A$ in a
domain of the complex plane is of interest in
scientific applications. When the matrix is real symmetric or complex hermitian, a procedure based on computations of Sturm sequences allows to safely apply bisections on real intervals to localize the eigenvalues. The problem is much harder for non symmetric or non hermitian matrices and especially for non normal ones. This last case is the main concern of this work. Proceeding by trying to compute the
eigenvalues of the matrix may not always be appropriate for two
reasons.

First most of the iterative methods
frequently used to calculate eigenvalues of large and sparse matrices may loose some of them, since only a part of the
spectrum is computed, and as such there is no guarantee to localize all the eigenvalues of the selected domain. When a shift-and-invert transformation is used, the eigenvalues are obtained in an order more or less dictated by their distance from the shift, and if one eigenvalue is skipped, there is no easy strategy that
allows to recover it.

Second
the entries of the matrix may be given with some errors and then the eigenvalues can only be localized in domains of $\IC$.

Many authors  have defined regions in the complex plane that include the eigenvalues of a given matrix. One of the main tool is the Gershgorin theorem. Since a straight application of the theorem often leads to large disks, some authors extended the family of inequalities for obtaining smaller regions by intersections which include eigenvalues (see e.g., \cite{BrMe94,HuZhSh06}). Other techniques consist to consider bounds involving the singular values (see e.g., \cite{Beattie2003281}), the eigenvalues of the hermitian part and the skew-hermitian part of the matrix (see e.g., \cite{AdTs06}), or the field of values of inverses of the shifted matrices (see e.g., \cite{HSZ08}).

For taking into account, possible perturbations of the matrix, Godunov\cite{Godunov92} and
Trefethen \cite{Tref92} have separately defined the notion of
the of $\epsilon$-spectrum or pseudospectrum of a matrix to address the problem.
The problem can then be
reformulated as that of determining level curves of the 2-norm of the resolvent $R(z)=(zI-A)^{-1}$ of the matrix $A$.

The previous approaches determine a priori enclosures of the eigenvalues. A dual approach can be considered: given some curve $(\Gamma)$ in the complex plane, count the number of eigenvalues of the matrix $A$ that are surrounded by $(\Gamma)$. This problem was considered in \cite{BerPhi01} where several procedures were proposed.
In this paper, we make some progress with respect to the work in \cite{BerPhi01}. Our procedure is based on the application of the residual theorem: the integration process leads to the evaluation of the principal argument of the logarithm of the function $g(z)=\det((z+h)I-A)/\det(zI-A) $. This function is also considered  in \cite{Bind08} to count the eigenvalues when a nonlinear eigenvalue problem is perturbed.

This work is mainly concerned with the control of the integration path so as to stay on the same branch along an interval when evaluating the principal argument of a logarithm.

In section \ref{sec:2}, we present
the mathematical tools. In section \ref{sec:3}, we present the basis of our strategy for following a branch of the logarithm function and conditions for controlling the path length. Section \ref{sec:4} deals with the implementation of our method: we show how to safely compute the determinant and how to include new points along the boundary. In section \ref{sec:5} we present numerical test results carried out on some test matrices and  in section \ref{sec:6}, we conclude with some few remarks and future works.

\section{Mathematical tool and previous works}
\label{sec:2}
In this section we present the Cauchy's argument principle and some previous works on counting eigenvalues in regions of the complex field.

\subsection{Use of the argument principle}
The localization of the eigenvalues of matrix $A$ involves the calculation of determinants. Indeed let $(\Gamma)$ be a closed piecewise regular Jordan curve (piecewise $C^{1}$ and of winding number 1) in the complex plane which does not include eigenvalues of $A$. The number $N_{\Gamma}$ of eigenvalues surrounded by $(\Gamma)$ can be expressed by the Cauchy formula (see e.g., \cite{Rudi70,Silverman}):
\begin{eqnarray}\label{cauchy}
    N_{\Gamma}&=& \frac{1}{2i\pi} \int_{\Gamma}{\frac{f^{'}(z)}{f(z)}dz},
\end{eqnarray}
where $f(z)=det(zI - A)$ is the characteristic polynomial of A. \newline \newline
If $\gamma (t)_{0\leq t \leq 1}$ is a parametrization of $\Gamma$ the equation (\ref{cauchy}) can be rewritten as
\begin{equation}
    N_{\Gamma} = \frac{1}{2i\pi}\int_{0}^{1} \frac{ f^{'} (\gamma (t))}{f (\gamma (t))} \gamma^{'} (t) dt .
\end{equation}
The primitive $\varphi$ defined by
\begin{equation*}
    \varphi (u) = \int_{0}^{u}\frac{ f^{'} (\gamma (t))}{f (\gamma (t))} \gamma^{'} (t) dt, u\in [0, 1],
\end{equation*}
is a continuous function which is a determination of
$\log (f \circ \gamma )$ (e.g. see \cite{Silverman}):
$$\log f( \gamma (t))= log \left|f( \gamma (t)) \right| + i \ arg (f( \gamma (t))),
\ \  t \in [0,1].$$
It then follows that
    $$N_{\Gamma} = \frac{1}{2\pi}\varphi_{I} (1),$$
where $\varphi _{I} (1)$ is the imaginary part of $\varphi (1)$ since its real part vanishes.

\subsection{Counting the eigenvalues in a region surrounded by a closed curve}
In \cite{BerPhi01}, two procedures were proposed for counting the
eigenvalues in a domain surrounded by a closed curve.

The first
method is based on the series expansion of $\log (I+h R(z))$, where $R(z)=(zI-A)^{-1}$, combined
with a path following technique.  The method uses a predictor - corrector scheme with constant step size satisfying the constraint
$$\left| \varphi _{I} (z+\Delta z)- \varphi _{I}(z) \right | < \pi,$$
for a discrete list of points $z$.
The implementation of the algorithm requires the computation of a few of the smallest singular values and the corresponding left and right singular vectors of $(zI -A)$; they are used to follow the tangent to the level curve of the smallest singular value of $(zI -A)$.

In the
second procedure, the domain is surrounded by a parameterized
user-defined curve $z=\gamma (t)$ and thus
\begin{equation}
    N_{\Gamma} = \frac{1}{2i\pi}\int_{\gamma (0)}^{\gamma (1)} \frac{\frac{d}{dt} det (\gamma (t)I-A)}{det (\gamma (t)I-A)}dt
\end{equation}
Since $\gamma (0)= \gamma (1)$, the function $\gamma (t)$ defined on $[0,1]$, can be extended onto $\IR$ by $$ \gamma_{ext}(t)=\gamma (t \text{ mod } 1).$$
By subdividing the interval $[ \gamma (0) , \gamma (1)]$ into
subintervals of equal length, and by assuming that $\gamma_{ext} \in {\cal C}^{\infty}$, they make use of a fundamental result from quadrature of periodic function to prove an exponential convergence of the integral.  The method is compared to other integrators with adaptive step sizes.

Each of these methods makes use of the computation of
$$u(t) = \frac{ det (\gamma (t)I-A)}{\left|det (\gamma (t)I-A)\right|},$$
which is efficiently computed through a LU factorization of the matrix $ ( \gamma (t)I-A)$ with partial
pivoting. In order to avoid underflow or overflow, the quantity is computed by
$$\frac{ det (\gamma (t)I-A)}{\left| det (\gamma (t)I-A)\right|}=\prod_{i=1}^{n}\frac{u_{ii}}{\left|u_{ii}\right|}$$
where $u_{ii}$ is the i-th diagonal element of U in the LU
factorization.  The product is computed using the procedure
that will be described later on in section \ref{sec:4}.

Our work, which can be viewed as an improvement of \cite{BerPhi01}, mostly deals with the control of the integration so as to
stay on the same branch along an interval, during the evaluation of the principal argument of the logarithm of the function g(z) defined in the introduction.

\section{Integrating along a curve}
\label{sec:3}
In this section, we describe strategies for the integration of the function
$g(z)=\frac{f^{'}(z)}{f(z)}$, where $f(z)=det(zI-A)$,
along the boundary of a domain limited by a
user-defined curve $(\Gamma)$ that does not include eigenvalues of $A$.

\subsection{Following a branch of log(f(z)) along the curve}
To simplify the presentation and without loss of generalization, let us assume that $ \Gamma = \bigcup_{i=0}^{N-1}{[z_i,z_{i+1}]} $ is a polygonal curve.

Let $\Arg(z)\in (- \pi , \pi]$
denote the principal determination of the argument of a complex number $z$, and
$\arg(z)\equiv \Arg(z) \ \ (2\pi) $, be any determination of the argument of $z$.
In this section, we are concerned with the problem of following a branch of $\log(f(z))$ when $z$ runs along $(\Gamma )$.
The branch (i.e. a determination $\arg_{0}$ of the argument), which is to be followed along the integrating process, is fixed by selecting an origin $z_{0} \in (\Gamma)$ and by insuring
\begin{equation}
\label{eq:branch0}
\arg_{0} (f(z_{0})) = \Arg(f(z_{0})).
\end{equation}
Let $z$ and $z+h$ two points of $(\Gamma)$. Since
\begin{eqnarray*}
(z+h)I-A &=&(zI-A)+hI \\
         &=& (zI-A)(I+hR(z)),
\end{eqnarray*}
where $R(z) = (zI-A)^{-1}$, it then follows that
\begin{equation}
\label{eq:fzh}
f(z+h)=f(z) \  \det(I+hR(z)).
\end{equation}
Let $\Phi_{z}(h) = \det(I+hR(z))$, then
\begin{eqnarray*}
\int_{z}^{z+h}\frac{ f^{'} (z)}{f(z)}dz &=& \log(f(z+h))- \log(f(z)) \\
&=& \log \left( \frac {f(z+h)}{f(z)}\right) \\
&=& \log (\Phi_{z}(h)) \\
&=&\log \left| \Phi_{z}(h) \right|+i \arg(\Phi_{z}(h)).
\end{eqnarray*}
In the previous approach \cite{BerPhi01}, given $z$, the step  $h$ is chosen such that condition
\begin{equation} \label{eq:phiargcond}
\left| \arg(\Phi_{z}(h)) \right| < \pi,
\end{equation}
is satisfied. In \cite{BerPhi01} condition $(\ref{eq:phiargcond})$ is only checked at point $z+h$ but we want the condition to be satisfied at all the points
$s \in [z,z+h]$, so as to guarantee that we stay on the same branch along the interval $[z,z+h]$.  We need a more restrictive condition
which is mathematically expressed by the following lemma:
\begin{lemma}[Condition (A)] \label{lema1}
Let $z$ and $h$ be such that $[z,z+h]\subset (\Gamma) .$ \\
If
\begin {equation}\label{condA}
\left| \Arg( \Phi_{z}(s))\right| < \pi, \ \ \forall s \in [0,h],
\end{equation}
then,
\begin{equation}
\label{eq:integr}
\arg_{0}(f(z+h))=\arg_{0}(f(z))+\Arg(\Phi_{z}(h)),
\end{equation}
where $\arg_{0}$ is the determination of the argument determined as in (\ref{eq:branch0}) by an a priori given $z_{0} \in (\Gamma)$.
\end{lemma}
\begin{proof}
We prove it by contradiction. Let us assume that there exists $k \in \mathbb{Z}\setminus \{0 \}$ such that
$$arg_{0}(f(z+h))=arg_{0}(f(z))+Arg(\Phi_{z}(h)) + 2 k\pi.$$
By continuity of the branch, there exists $s \in [0,h]$ such that $\left|Arg( \Phi_{z}(s))\right| = \pi$, which contradicts $(\ref{condA}).$
\end{proof}
Condition $(\ref{condA})$ is called {\bf Condition (A)}.

\subsection{Step size control}
In our approach, given $z$, the step $h$ is chosen such that condition of Lemma $(\ref{lema1})$ is satisfied.
{\bf Condition (A)} is equivalent to $$\Phi_{z}(s) \notin (- \infty , 0], \forall s \in [0 , h].$$
In order to find a practical criterion to insure it, we look for a more severe condition
by requiring that $\Phi_{z}(s) \in \Omega$, where $\Omega$ is an open convex set, neighborhood of $1$, and included in  $\Omega \subset \mathbb{C}\setminus(- \infty , 0]$. Possible options for $\Omega$ are
the positive real half-plane, or any disk included in it and centered in $1$.

Since $\Phi_{z}(0)=1$, let $$ \Phi_{z}(s) = 1+ \delta , \mbox{ with } \delta =\rho e^{i\theta}.$$ A sufficient condition for $(\ref{condA})$ be to satisfied is
$\rho <1$, i.e.
\begin{equation}
\left|\Phi_{z}(s) - 1 \right| <1 , \ \ \forall s \in [0,h] \label{eq:condB}
\end{equation}
This condition will be referred to as {\bf Condition (B)}, and, when only verified at $z+h$, i.e.
\begin{equation}
\left|\Phi_{z}(h) - 1 \right| <1 , \label{eq:condBprime}
\end{equation}
it will be referred to as {\bf Condition (B')}. This last condition is the condition used in \cite{BerPhi01}. It is clear that {\bf Condition (B)} implies {\bf Condition (A)} whereas this is not the case for {\bf Condition (B')}.

Since it is very difficult to check (\ref{eq:condB}), we apply the condition on the linear approximation
$\Psi_{z}(s)=1 +s \Phi_{z}^{'}(0)$ of $\Phi_{z}(s)$ at $0$.
Replacing function $\Phi_z$ by its tangent $\Psi_z$ in (\ref{eq:condB}), leads to
\begin{equation}
\left | \Psi_{z}(s) - 1 \right | <1, \ \ \forall s \in [0,h], \label{condBlinear}
\end{equation}
which is equivalent to the following condition, referred as {\bf Condition (C)}:
\begin{equation}
\left|h\right|< \frac{1}{|\Phi_{z}^{'}(0)|}. \label{condC}
\end{equation}

\begin{example}[First illustration]\label{ex:1}
Let
$ A= \left (
        \begin{array}{cc}
            0& 0 \\
            0& 1
        \end{array}
    \right ).$
It then follows that
\begin{eqnarray*}
f(z) &=& z(z-1),\\ \Phi_{z}(h)&=&(1+\frac{h}{z})(1+\frac{h}{z-1}),\\
\Phi_{z}^{'}(0)&=&\frac{1}{z}+\frac{1}{z-1}.
\end{eqnarray*}
Let us assume that we are willing to integrate along the segment from $z=2$ to $z=1+i$. In order to see if intermediate points are needed to insure that the branch of the logarithm is correctly followed, we consider the previously introduced conditions on $h=t(-1+i)$ where $t \in [0,1]$.

\begin{description}
\item [Condition (A):] $ \Phi_{2}(h)=1+\frac{3h}{2}+\frac{h^2}{2}$ is a non positive real number if and only if $h\in [-2,-1] \bigcup ( -\frac{3}{2}+i\IR )$.  From that, it can easily be seen that the segment $[0, -1+i]$ does not intersect the forbidden region. Therefore no intermediate points are needed.
\item [Condition (B):] this condition is equivalent to $|h| |3+h| < 2$.  By studying the function $\phi (t) = |h| |3+h| = \sqrt{2} t |3-t+it|$, the parameter $t$ must remain smaller than $\alpha \approx 0.566$.
\item [Condition (B'):] in this example, this condition is equivalent to the previous one, since the function $\phi (t)$ is increasing with $t$.
\item [Condition (C):] since $\Phi_{2}^{'}(h)=\frac{3}{2}+h$, this condition limits the extent of the interval to $|h|<\frac{2}{3}$ or equivalently $t < \frac{\sqrt{2}}{3} \approx 0.471$ .
\end{description}
\end{example}

In the second example, we illustrate the lack of reliability of {\bf Condition (B')}. 

\begin{example}[Second illustration] \label{ex:1bis}
Let
$ A= \lambda I_n$, where $\lambda \in \IR$ and $I_n$ is the identity matrix of order $n$.
It then follows that
\begin{eqnarray*}
f(z) &=& (z-\lambda)^n,\\ \Phi_{z}(h)&=&\left(1+\frac{h}{z-\lambda}\right)^n,\\
\Phi_{z}^{'}(0)&=&\frac{n}{z-\lambda}.
\end{eqnarray*}
Let us assume that we are willing to integrate from $z=\lambda+1$ to $z+h=\lambda+e^{i\theta}$. We consider the previously introduced conditions on $h$.
\begin{description}
\item [Condition (A):]  $|\theta | < \frac{\pi}{n}$.
\item [Condition (B):]  $|\theta| < \frac{\pi}{3n}$.
\item [Condition (B'):]  $\cos n\theta > \frac{1}{2}$ which is satisfied for values that violate {\bf Condition (A)}.
\item [Condition (C):]  $|\frac{\theta}{2}|< \arcsin\frac{1}{2n}$, which is more severe than $|\theta| < \frac{1}{n}$ and therefore guaranties {\bf Condition (A)}.
\end{description}
\end{example}
In this example, if $(\Gamma)$ is the circle with center $\lambda$ and radius $1$, the step size must be reduced in such a way that more than $2n$ intervals are considered to satisfy {\bf Condition (A)}, or even $6n$ and $2\pi n$ intervals with 
{\bf Condition (B)} and {\bf Condition (C)} respectively.

Practically, we consider that {\bf Condition (C)} implies {\bf Condition (A)}, as long as the linear approximation is valid. Problems may occur when $\Phi_{z}^{'}$ vanishes. The following example illustrates such a situation.

\begin{example}[Critical situation] \label{ex:2}
Let us consider the matrix of Example \ref{ex:1}.
For $z=1/2$, $\Phi_{1/2}(h)=1-4h^{2}$, and $\Phi_{1/2}^{'}(0)=0$, and
the conditions become
\begin{description}
\item [Condition (A):] $h \notin \IR$ or $|h| < 1/2$,
\item [Condition (B):]  $|h|<1/2$,
\item [Condition (B'):]  $|h|<1/2$,
\item [Condition (C):] is satisfied for all $h \in \IC$.
\end{description}
\end{example}

\section{Implementation}
\label{sec:4}
In this section, we describe the numerical implementation of our method.  Strategies for including new points and a procedure for safely computing the determinants are given.

\subsection{Avoiding overflows and underflows}
The implementation of our method requires the computation of $$\Phi_{z}(h) = \frac{\det((z+h)-A)}{\det(zI-A)}.$$
In order to avoid underflow or overflow, we proceed as follows: \newline \newline
For any non singular matrix $M\in \IC^{n\times n}$, let us consider its LU
factorization $PM=LU$ where $P$ is a permutation matrix of signature $\sigma$. Then $\det(M)=\sigma \prod_{i=1}^{n}(u_{ii})$ where
$u_{ii}\in \mathbb{C}$ are the diagonal entries of $U$. If the matrix $M$ is not correctly scaled, the product
$\prod_{i=1}^{n}(u_{ii})$ may generate an overflow or an underflow.
To avoid this, the determinant is characterized by the triplet $( \rho,
K, n )$ so that
\begin{equation}\label{det2}
    \det(A)=\rho K^{n}
\end{equation}
where:
\begin{eqnarray*}
\rho & = & \sigma \prod_{i=1}^{n}\frac{u_{ii}}{\left|u_{ii}\right|},
\ \ (\rho \in \mathbb{C} \mbox{ with } \left|\rho \right|=1), \mbox{ and } \\
 K &=& \sqrt[n]{\prod_{i=1}^{n}|u_{ii}|} \ \ ( K > 0).
\end{eqnarray*}
The quantity $K$ is computed through its logarithm:
$$ \log(K) = \frac{1}{n} \sum_{i=1}^n \log (|u_{ii}|). $$
By this way, the exact value of the determinant is not computed, as long as the scaling of the matrix is not adequate.

In section \ref{sec:3}, it was indicated that our algorithm will heavily be based on the computation of $\Phi_{z}(h)= \det (I+hR(z)).$
For $h$ of moderate modulus, the determinant does not overflow.  This can be verified since $$\Phi_{z}(h) = \frac{\det((z+h)I-A)}{\det(zI-A)} = \frac {K_{2}^{n} \rho_{2}}{K_{1}^{n} \rho_{1}},$$
where $\det(zI-A)$ and $\det((z+h)I-A)$ are respectively represented by the triplets $(\rho_1,K_1,n)$ and $(\rho_2,K_2,n)$.
Before raising to power $n$, to protect from under- or overflow, the ratio $K_{2}/K_{1}$ must be in the interval $[ \frac{1}{\sqrt[n]{M_{fl}}}, \sqrt[n]{M_{fl}}]$  where $M_{fl}$ is the largest floating point number.  When this situation is violated, intermediate  points must be inserted between $z$ and $z+h$.

\subsection{Estimating the derivative}

An easy computation shows that the derivative $\Phi_{z}^{'}(0)$ can be expressed by :
\begin{equation}
\label{eq:trace}
\Phi_{z}^{'}(0) = {\rm trace}(R(z)).
\end{equation}
The evaluation of this simple expression involves many operations, as we show it now. By using the LU-factorization $P(zI-A)=LU$ which is available at $z$, and by using (\ref{eq:trace}), we may compute $\Phi_{z}^{'}(0) = \sum_{i=1}^n u_i^*l_i$, where $l_i=L^{-1}e_i$ and $u_i=(U^*)^{-1}e_i$, with $e_i$ being the $i$-th column of the identity matrix. When $A$ is a sparse matrix, the factors $L$ and $U$ are sparse but not the vectors $u_i$ and $l_i$. 
Therefore, the whole computation involves $2n$ sparse triangular systems. Experiments showed that they involve more operations than the LU factorization of the matrix. Approximations of the trace of the inverse of a matrix have been investigated. 
They involve less operations than use of the LU factorization but they are only valid for symmetric or hermitian matrices \cite{Bai96,GoMe09}.

If the derivative is approximated by its first order approximation, sparsity helps. More specifically, given $z$ and $z+h$, the derivative of $\Phi_{z}$ at $0$ is estimated by
$$\Phi_{z}^{'}(0) \approx \frac{\Phi_{z}(s) - 1}{s}, $$ where $s=\alpha h$ with $\alpha= \min(10^{-6}\mu/|h|,1) ,$ and $\mu = \max_{z\in\Gamma}{\left|z\right|}$. Therefore, the computation imposes an additional LU factorization for evaluating the quantity $\Phi_{z}(s)$. It is known that, for a sparse matrix, the sparse LU factorization involves much less operations that its dense counterpart.

\subsection{Test for including new points} \label{sec:newpts}
In this subsection, we describe a heuristic procedure for including new points in the interval $[z,z+h]$. In section {\ref{sec:3},
we introduced {\bf Condition (B)} which is more severe than {\bf Condition (A)} but might be easier to verify, and we proposed to test its linear
approximation called {\bf Condition (C)}.  Unfortunately Example \ref{ex:2} has exhibited that {\bf Condition (C)} may be satisfied
while {\bf Condition (B')} and therefore {\bf Condition (B)} is violated.  To increase our confidence in accepting the point $z+h$, we
simultaneously check {\bf Condition (C)} and {\bf Condition (B')}.

When {\bf Condition (C)} is violated, we insert $M$ regularly spaced points between $z$ and $z+h$ where
\begin{equation} \label{eq:ptsinsert}
M=\min\left(\left\lceil |h| \ |\Phi_{z}^{'}(0)| \right\rceil, M_{max} \right),
\end{equation}
with $M_{max}$ being some user defined parameter.

In addition, we insist that {\bf Condition (C)} is satisfied at each bound of the segment $[z,z+h]$. Therefore, on exit, the condition $\left|h\right|< \frac{1}{|\Phi_{z+h}^{'}(0)|}$ must also be guarantied. When it is violated, we insert the point $z+h/2$ in the list.

The following example illustrates the effect of this step size control.

\begin{example} \label{ex:test1}
Let $A$ be the random matrix :
\begin{eqnarray*}
A & = & \left( \begin{array}{rrrrr}
 -0.63 & 0.80 & 0.68 & 0.71 & -0.31 \\
 -0.81 & 0.44 & -0.94 &  0.16 &  0.93 \\
 0.75 & -0.09 & -0.91 & -0.83 & -0.70 \\
 -0.83 & -0.92 &  0.03 & -0.58 & -0.87 \\
 -0.26 & -0.93 & -0.60 & -0.92 & -0.36
 \end{array} \right) .
 \end{eqnarray*}
 The polygonal line $(\Gamma)$ is determined by 10 points regularly spaced on the circle of center 0 and radius 1.3. In Figure \ref{fig:ex1}, are displayed the eigenvalues of $A$, the line $(\Gamma)$ and the points that are automatically inserted by the procedure. The figure illustrates that, when the line gets closer to some eigenvalue, the segment length becomes smaller.
\end{example}

\begin{figure}
\centering
\includegraphics[width=13cm]{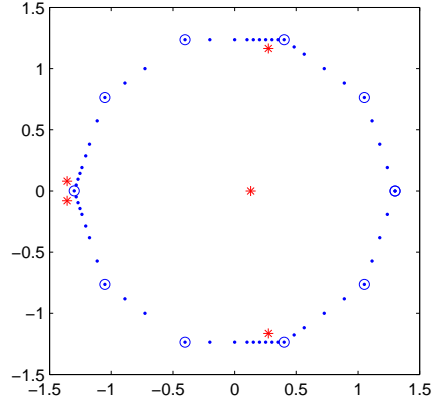}
\caption{\label{fig:ex1} Example \ref{ex:test1}. {\it The eigenvalues are indicated by the stars. The polygonal line is defined by the 10 points with circles; the other points of the line are automatically introduced to insure the conditions as specified in section \ref{sec:newpts}} ($M_{max}=1$ in (\ref{eq:ptsinsert})).}
\end{figure}

\subsection{Global algorithm}

The algorithm is sketched in Table \ref{algo:eigcnt}. From a first list $Z$ of points, it extends the list $Z$ in order to determine a safe split of the integral (\ref{cauchy}). The complexity of the algorithm is based on the number of computed determinants. For each $z \in Z$, the numbers $det(zI-A)$ and $\Phi_{z}^{'}(0)$ are computed; they involve two evaluations of the determinant. Therefore, for $N$ final points in $Z$, the complexity can be expressed by:
$$ {\cal C} = 2 {\cal L}_{LU} N, $$
where ${\cal L}_{LU}$ is the number of operations involved in the complex LU factorization of $zI-A$.

When the matrix $A$ is real and, assuming that the polygonal line $(\Gamma)$ is symmetric w.r.t. the real axis and intersects it only in two points, half of the computation can be saved since
$$  N_{\Gamma} = \frac{1}{\pi} \  {\cal I}\left( \int_{\Gamma_+}\frac{f^{'}(z)}{f(z)}dz \right),$$
where $(\Gamma_+)$ is the upper part of $(\Gamma)$ when split by the real axis, and ${\cal I}(Z)$ denotes the imaginary part of $Z$.

\begin{table}
\centering
\begin{ualgorithm*}{\sc Eigencnt}
\> \[Input] \\
\> \> $Z$=\{edges of $(\Gamma)$\} ; \\
\> \> $M_{pts}=\mbox{maximum number of allowed points};$ \\
\> \> $M_{max}=\mbox{maximum number of points to insert simultaneously} $ ;\\
\> \[Output] \\
\> \> $neg$ = number of eigenvalues surrounded by $(\Gamma)$ ; \\ \\
\> Status($Z$)=-1 ; \\
\> \[while] Status($Z$)$\neq$ 0 \[and] length($Z$) < $M_{pts}$, \\
\> \> \[for] $z \in Z$ such that Status($z$)==-1, \\
\> \> \> Compute $det(zI-A)$ and  $\Phi_{z}^{'}(0)$ ; \\
\> \> \> Status($z$) = 1 ; \\
\> \> \[end] \\
\> \> \[for] $z \in Z$ such that Status($z$)=1, \\
\> \> \> \[if] {\bf Condition (C)} not satisfied at $z$, \\
\> \> \> \> Generate $M$ points $\tilde{Z}$ as in (\ref{eq:ptsinsert}); \\
\> \> \> \> $Z$=$Z\cup \tilde{Z}$; Status($\tilde{Z}$)=-1; \\
\> \> \> \[elseif] {\bf Condition (B')} not satisfied at $z+h$ ; \\
\> \> \> \> $Z$=$Z\cup \{ z+h/2 \}$; Status($z+h/2$)=-1; \\
\> \> \> \[else] \\
\> \> \> \> Status($z$)=0 ; \\
\> \> \> \[end] \\
\> \> \[end] \\
\> \> \[if] no new points were inserted in $Z$ ; \\
\> \> \> \[for] $z \in Z$, \\
\> \> \> \> \[if] {\bf Condition (C)} is backwardly violated ; \\
\> \> \> \> \> $Z$=$Z\cup \{ z-h/2 \}$; Status($z-h/2$)=-1; \\
\> \> \> \[end] \\
\> \> \[end] \\
\> \[end] \\

\> Integral = $\sum_{z \in Z} \Arg (\Phi_z ^{'}(0))$ ; $neg$ = round(Integral/ $2 \pi$)  ;
\end{ualgorithm*}
\caption{\label{algo:eigcnt}Algorithm for counting the eigenvalues surrounded by $(\Gamma)$.}
\end{table}

%
\section{Numerical tests}
\label{sec:5}
The tests are run on a laptop Dell (Processor Intel Core i7-2620M CPU, clock: 2.70 GHz, RAM: 4 GB). The program {\tt eigencnt} is coded in Matlab.

In the following tests, we describe the performances of the algorithm for three real matrices chosen from the set Matrix Market \cite{Market}. The maximum inserted points in an interval is $M_{max}=10$. When $(\Gamma)$ is symmetric w.r.t. the real axis, only half of the integration is performed. The storage of the matrices is kept sparse (except for computing the spectra of the matrices of the two first examples).

\begin{example}[Matrix ODEP400A]
\label{ex:odep}
This matrix is a model eigenvalue problem coming from an ODE.
\end{example}
\begin{center}
\begin{tabular}{|c|c|c|c|}
\hline
$n$ & $\| A \|_1$ & Spectral radius & Spectrum included in \\ \hline
400 & 7 & 4.00 & [-4,4.38e-4]x[-0.01,0.01] \\ \hline
\end{tabular}
\end{center}
This matrix is of small order and its spectrum is displayed in Figure \ref{fig:spect-odep400a}.
\begin{figure}
\centering
\includegraphics[width=10cm]{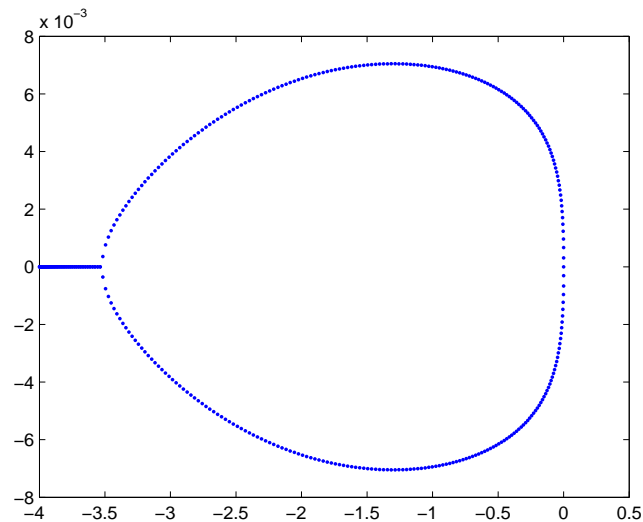}
\caption{\label{fig:spect-odep400a}Spectrum of the matrix of Example \ref{ex:odep}.}
\end{figure}

\begin{figure}[hbt]
\centering
\includegraphics[width=10cm]{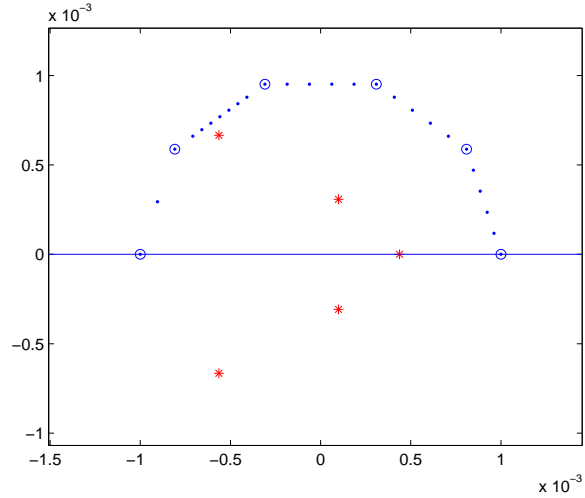}
\caption{\label{fig:right-odep400a} Example \ref{ex:odep}: First experiment on the right end of the spectrum.}
\end{figure}

The first experiment consists to focusing on the right part of the spectrum by defining a regular polygon of 10 vertices; the polygon is centered in the origin,  symmetric w.r.t. with the real axis as shown in figure \ref{fig:right-odep400a} (only its upper part is drawn), and of radius $R=10^{-3}$. Five eigenvalues were correctly found as surrounded by the polygon. Some statistics are displayed in the first line of Table \ref{tab:odep}.
\begin{table}[ht]
\caption{\label{tab:odep}Statistics for Example \ref{ex:odep}.}
\centering
\begin{tabular}{|c|c|c|c|}
\hline
 & Nb. of eigenvalues in $(\Gamma)$ & nb. of intervals & elapsed time \\ \hline
Exper. 1 & 5 & 25 & 5.6e-2 s \\ \hline
Exper. 2 & 89 &  1519 & 1.1 s \\ \hline
\end{tabular}
\end{table}

The second experiment focuses on the bifurcation between real and complex eigenvalues in the neighborhood of $-3.5$. In the box $[-4,-3.4]\times [-10^{-4}i,10^{-4}i]$, $89$ eigenvalues are counted (see the statistics in the second line of Table \ref{tab:odep}). The aspect ratio of the box is large. The refining process proceeds in 16 steps to produce 1519 intervals from the initial four. If the integral is computed by the relation (\ref{eq:integr}) at each step (hence even if the necessary conditions for correctness are not satisfied), it would only have been correct at the fifth step and after ; this corresponds to 825 intervals. This illustrates the loss in efficiency which is imposed by the constraint for a safe computation.

\begin{example}[Matrix TOLS2000]
\label{ex:tols}
This matrix comes from a stability analysis of a model of an airplane in flight.
\end{example}
\begin{center}
\begin{tabular}{|c|c|c|c|}
\hline
$n$ & $\| A \|_1$ & Spectral radius & Spectrum included in \\ \hline
2000 & 5.96 $\times 10^6$ &  r= 2.44 $\times 10^3$& [-750,0]x[-r,+r]  \\ \hline
\end{tabular}
\end{center}

\begin{figure}[ht]
\hspace{-1.5cm} ~
\begin{tabular}{cc}
\multicolumn{2}{c}{\includegraphics[width=10cm]{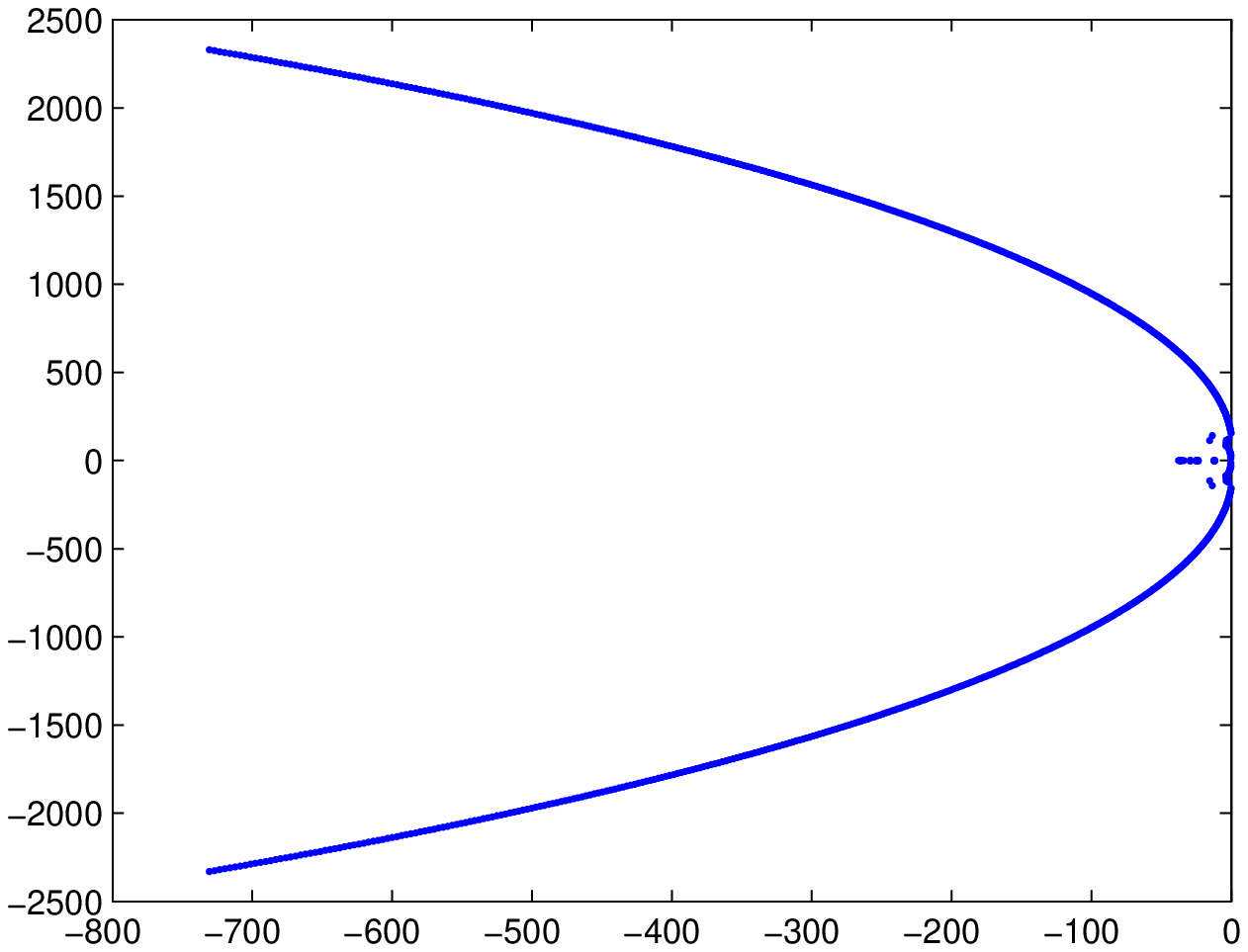}} \\ ~\\
\multicolumn{2}{c}{Entire spectrum of the matrix TOLS2000} \\ ~\\
\includegraphics[width=7cm]{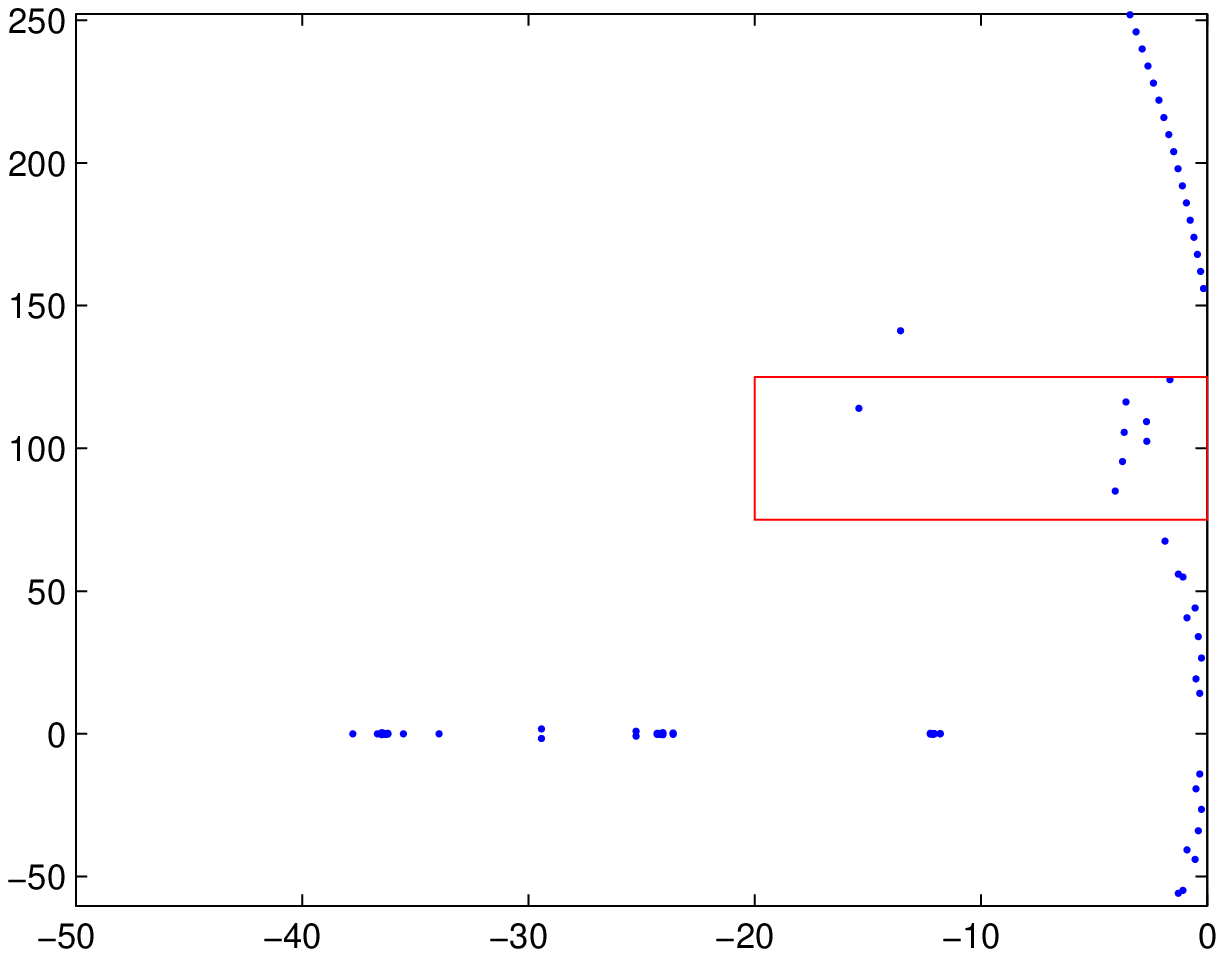} &
\includegraphics[width=7cm]{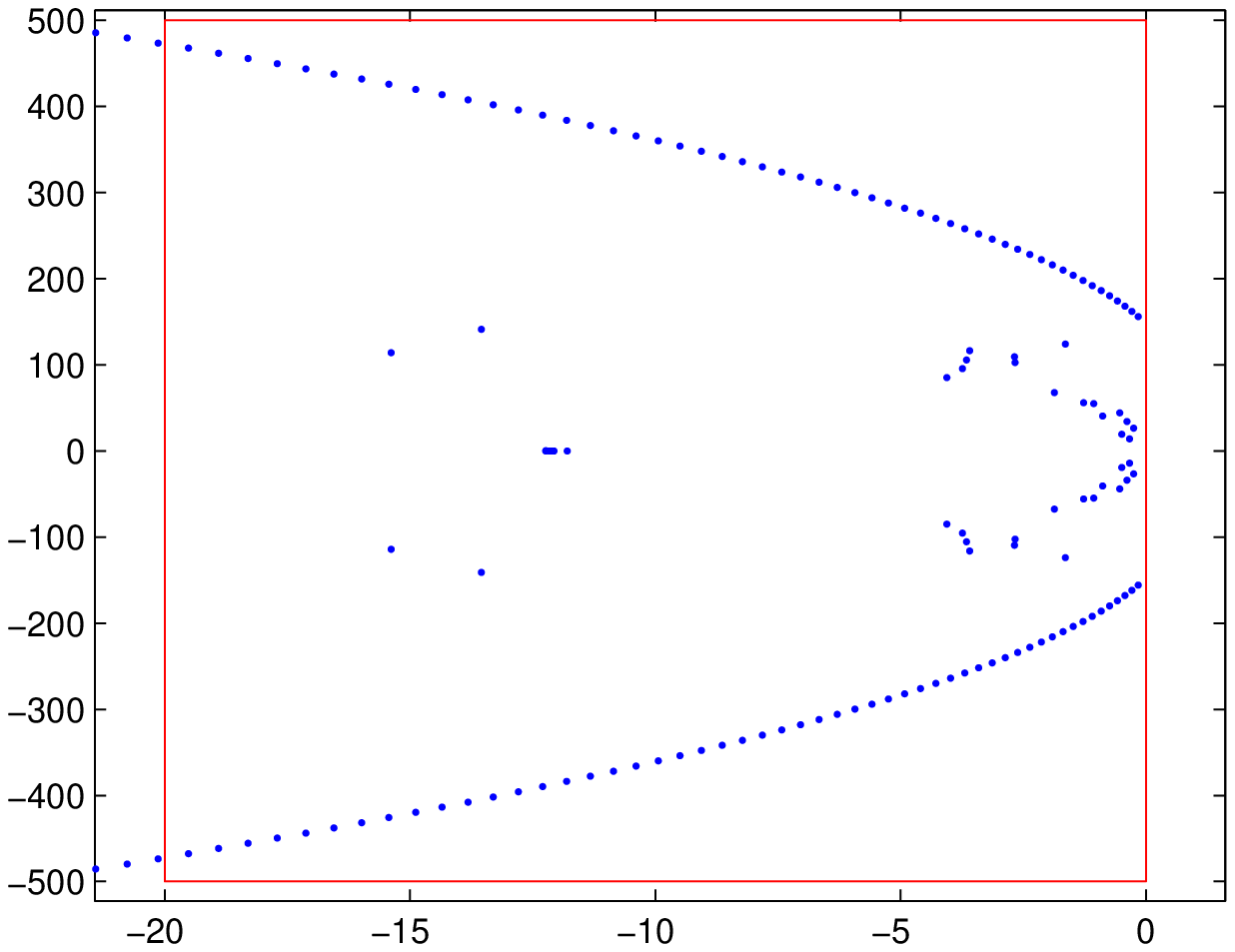} \\ ~& \\
{ Exper. 1}: Box=$[-20,0]\times [75i,125i]$ & { Exper. 2}: Box=$[-20,0]\times [-500i,500i]$
\end{tabular}
\caption{\label{fig:tols} Example \ref{ex:tols}: Spectrum of the matrix (up) and zooms on the two regions of experiments.}
\end{figure}

Two experiments consider the right part of the spectrum. A first box
$[-20,0]\times [75i,125i]$ is not symmetric w.r.t. the real axis. Therefore, the integration is not reduced. Eight eigenvalues are decounted. The second box $[-20,0]\times [-500i,500i]$ is symmetric w.r.t. the real axis but it includes 542 eigenvalues. The statistics are reported in Table \ref{tab:tols}. In Figure \ref{fig:tols}, the spectrum and two zooms on it are displayed.

\begin{table}[ht!]
\centering
\begin{tabular}{|c|c|c|c|}
\hline
 & Nb. of eigenvalues in $(\Gamma)$ & nb. of intervals & elapsed time \\ \hline
Exper. 1 & 8 &  2611 & 11.0 s \\ \hline
Exper. 2 & 542 & 15669 & 57.7 s \\ \hline
\end{tabular}
\caption{\label{tab:tols}Statistics for Example \ref{ex:tols}.}
\end{table}
\newpage
\begin{example}[Matrix E40R5000]
\label{ex:e40r5000}
This sparse matrix comes from modeling 2D fluid flow in a driven cavity, discretized on a $40\times 40$ grid and with a Reynolds number is $Re=5000$.
\end{example}
\begin{center}
\begin{tabular}{|c|c|c|c|}
\hline
$n$ & $\| A \|_1$ & Spectral radius & Spectrum included in \\ \hline
17281 &  1.21 $\times 10^2$ &  r=65.5 (*) & [0,20.2]x[-r,+r]  \\ \hline
\end{tabular}
\end{center}
(*) Estimated by the Matlab procedure {\tt eigs}.

This example shows the reliability of the proposed procedure.
Computing the 6 eigenvalues of largest real part with the Matlab procedure {\tt eigs} (it implements the ARPACK code) returns the eigenvalues $\{ 8.3713 \pm 64.653i, 8.8025 \pm 64.876i , 16.203 , 20.166 \}$. By increasing the number $p$ of requested eigenvalues, only a few of them converged: for instance for $p=20$, only the two rightmost were found. Increasing even further up to $p=100$, 14 eigenvalues were given back, including the already computed two rightmost and 12 additional ones with real parts belonging to the interval [12.2,12.9]. Therefore, the user is inclined to ask for the exact situation in this region. Defining the rectangle
$\Gamma=\Gamma_+ \cap \Gamma_-$ where $\Gamma_+ =(14,14+2i,12+2i,12)$ and where $\Gamma_-$  is the symmetric of $\Gamma_+$ w.r.t. the real axis, the procedure {\tt eigencnt} returns
\begin{table}[hbt]
\centering
\begin{tabular}{|c|c|c|}
\hline
Number of eigenvalues in $(\Gamma)$ & number of intervals & elapsed time \\ \hline
116 & 7986 & 54 h 42 mn \\ \hline
\end{tabular}
\caption{\label{tab:e40r5000}Statistics for Example \ref{ex:e40r5000}.}
\end{table}

Actually, the right number of eigenvalues was already given before the last refining step with $3994$ intervals.
Taking into account the result of the experiment, after several tries of shifts in {\tt eigs}, all the 116 eigenvalues surrounded by $(\Gamma)$ were obtained by requesting $p=200$ eigenvalues in the neighborhood of the shift $\sigma=13.5$ (elapsed time: 10.2s).

\section{Conclusion}
\label{sec:6}
In this paper, we have developed a reliable method for counting the eigenvalues in a region surrounded by a user-defined polygonal line. The main difficulty to tackle lies in the step control which must be used during the complex integration along the line. The method is reliable but it involves a high level of computation. This is the price to pay for the reliability. In forthcoming works, a parallel version of the method will be developed and implemented. The code involves a high potential for parallelism since most of the determinant computations are independent. A second level of parallelism is also investigated within the computation of a determinant for matrices arising in domain decompositions.

\bibliographystyle{plain}
\bibliography{biblioKP}
\newpage
\tableofcontents

\end{document}